\documentclass[11pt, reqno]{amsart}

\usepackage[utf8]{inputenc}
\usepackage[T1]{fontenc}
\usepackage{amsthm}
\usepackage{amsmath}
\usepackage{amssymb}
\usepackage[inline]{enumitem}
\usepackage[dvips]{graphicx}
\usepackage{color}
\usepackage{hyperref}
\usepackage{url}
\usepackage{stackrel}
\usepackage[left=3.5cm, right=3.5cm, paperheight=11.8in]{geometry}
\usepackage{fancyhdr}
\usepackage{mathrsfs}
\usepackage{stmaryrd}
\usepackage{soul}
\usepackage{nicefrac}
\usepackage{datetime}
\usepackage{comment}
\usepackage[normalem]{ulem}
\longdate

\AtBeginDocument{%
   \def\MR#1{}
}

\newtheorem{theorem}{Theorem}[section]
\newtheorem{lemma}[theorem]{Lemma}
\newtheorem{proposition}[theorem]{Proposition}
\newtheorem{corollary}[theorem]{Corollary}
\newtheorem{definition}[theorem]{Definition}

\theoremstyle{definition}

\newtheorem{remark}[theorem]{\textbf{Remark}}

\newtheorem{question}[theorem]{\textbf{Question}}

\theoremstyle{remark}

\newtheorem*{claim*}{\textsc{Claim}}

\hypersetup{
    pdftitle={Small sets},
    pdfauthor={Paolo Leonetti and Salvatore Tringali},
    pdfmenubar=false,
    pdffitwindow=true,
    pdfstartview=FitH,
    colorlinks=true,
    linkcolor=blue,
    citecolor=green,
    urlcolor=cyan
}

\renewcommand{\bf}{\mathbf}
\renewcommand{\emptyset}{\varnothing}
\renewcommand{\rho}{\varrho}
\renewcommand{\ast}{\star}

\providecommand{\HHb}{\mathbf{H}}
\providecommand{\AAc}{\mathscr{A}}

\providecommand{\dd}{\mathsf{d}}
\providecommand{\NNb}{\mathbf{N}}
\providecommand{\PPb}{\mathbf{P}}
\providecommand{\PPc}{\mathcal{P}}

\providecommand{\RRb}{\mathbf{R}}
\providecommand{\ZZb}{\mathbf{Z}}

\providecommand\llb{\llbracket}
\providecommand\rrb{\rrbracket}

\begin{document}

\title{On small sets of integers}

\author{Paolo Leonetti}
\address{Institute of Analysis and Number Theory, Graz University of Technology | Kopernikusgasse 24/II, 8010 Graz, Austria}
\curraddr{Department of Statistics, Universit\`a ``Luigi Bocconi'' | via Roentgen 1, 20136 Milano, Italy} 
\email{leonetti.paolo@gmail.com}
\urladdr{https://sites.google.com/site/leonettipaolo/}

\author{Salvatore Tringali}
\address{School of Mathematical Sciences,
	Hebei Normal University | Shijiazhuang, Hebei province, 050024 China}
\email{salvo.tringali@gmail.com}
\urladdr{http://imsc.uni-graz.at/tringali}

\thanks{P.L. was supported by the Austrian Science Fund (FWF), project F5512-N26.}



\begin{abstract}
\noindent{}An upper quasi-density on $\bf H$ (the integers or the non-negative integers) is a real-valued subadditive function $\mu^\ast$ defined on the whole power set of $\mathbf H$ such that $\mu^\ast(X) \le \mu^\ast({\bf H}) = 1$ and $\mu^\ast(k \cdot X + h) = \frac{1}{k}\, \mu^\ast(X)$ for all $X \subseteq \bf H$, $k \in {\bf N}^+$, and $h \in \bf N$, where $k \cdot X := \{kx: x \in X\}$; and an upper density on $\bf H$ is an upper quasi-density on $\bf H$ that is non-decreasing with respect to inclusion. We say that a set $X \subseteq \bf H$ is small if $\mu^\ast(X) = 0$ for every upper quasi-density $\mu^\ast$ on $\bf H$.

Main examples of upper densities are given by the upper analytic, upper Banach, upper Buck, and upper P\'olya densities, along with the uncountable family of upper $\alpha$-densities, where $\alpha$ is a real parameter $\ge -1$ (most notably, $\alpha = -1$ corresponds to the upper logarithmic density, and $\alpha = 0$ to the upper asymptotic density).
 
It turns out that a subset of $\bf H$ is small if and only if it belongs to the zero set of the upper Buck density on $\bf Z$. This allows us to show that many interesting sets are small, including the integers with less than a fixed number of prime factors, counted with multiplicity; the numbers represented by a binary quadratic form with integer coefficients whose discriminant is not a perfect square; and the image of $\bf Z$ through a non-linear integral polynomial in one variable.
\end{abstract}
\maketitle
\thispagestyle{empty}

\section{Introduction}\label{sec:introduction}

It is not infrequently the case in number theory and other fields that one is faced with the problem of determining whether a given set of integers is ``small'' in a suitable sense, driven by the idea that ``largeness implies structure'' and with the proviso that if a set is ``small'' then its complement is ``large''.
In such cases, one is not necessarily interested in ``quantifying the largeness'' of a set $X$, but only in establishing whether $X$ is large. A classic approach is to let the collection of all ``small sets'' be an \emph{ideal}, i.e., a family $\mathcal I$ of proper subsets of a certain ``ambient set'' such that $\mathcal I$ is closed under taking subsets and finite unions.
In fact, many natural ideals on $\mathbf N$ can be represented as the inverse image of $0$ through an appropriate ``measure of largeness'' $f: \mathcal{P}(\mathbf N) \to \mathbf{R}$, cf.~\cite{Sol99} and see \S~\ref{sec:notation} for notation.
This has eventually led to the study of a diversity of set functions that, while retaining fundamental features of measures, are better suited than measures to certain applications;
some of these ``surrogate measures'', recently considered in \cite{LT, LT17} and called \emph{upper quasi-densities} (Definition \ref{def:upperdensity}), are also the subject of the present work. 

In detail, the plan of the paper is as follows.
In \S~\ref{sec:upper_densities}, we introduce the notion of ``small set'' specifically used in the present work (Definition \ref{def:Dsmall}) and characterize small sets in terms of the zero set of the upper Buck density (Theorem \ref{thm:2.4}), a little known ``upper density'', first considered in \cite{Bu0}, that happens to play a central role in the theory.
Then, in \S\S~\ref{sec:preliminaries} and \ref{sec:homogeneous}, we prove that the following sets are small: the integers with less than a fixed number of prime factors, counted with multiplicity (Corollary \ref{cor:upper_density_of_primes}); 
the non-negative integers whose base-$b$ representation does not contain a given non-empty string of digits (Corollary \ref{thm:digits}); 
the image of $\mathbf{Z}$ through a non-linear integer polynomial in one variable (Theorem \ref{thm:polynomialandDsmallsets}); 
the inverse image of the primes under a non-constant integral polynomial in one variable (Theorem \ref{thm:polynomialprimeandDsmall}); 
the numbers represented by a binary quadratic form with integer coefficients whose discriminant is not a perfect square (Theorem \ref{thm:4.2}). 
We conclude in \S~\ref{sec:closings} with some questions and remarks.

All in all, we can thus extend and unify several ``independent results''
so far only known for some of the classic ``upper densities'' encountered in the literature. In particular, Theorem \ref{thm:4.2} generalizes the well-known fact that the asymptotic density of the non-negative integers representable by a sum of two squares is zero.

\subsection{Generalities}
\label{sec:notation}

We refer to \cite{LT} for most notation, terminology, and conventions used through this paper. In particular, we write $\RRb$, $\ZZb$, and $\NNb$, resp., for the sets of reals, integers, and non-negative integers;
and unless noted otherwise, we reserve the letters $h$, $i$, and $k$ (with or without subscripts) for non-negative integers, the letters $m$ and $n$ for positive integers, and the letter $s$ for a real number. 

We let $\mathbf{P} \subseteq \bf Z$ be the set of (positive or negative) primes, and for all $k \in \mathbf Z$ and $m \in \mathbf N^+$ we denote by $k \bmod m$ the smallest non-negative integer $r$ such that $k \equiv r \bmod m$.
For all $a,b \in \RRb \cup \{\pm\infty\}$, we take $\llb a,b\rrb:=[a,b] \cap \ZZb$; and for every $X\subseteq \mathbf{R}$ and $h, k \in \RRb$, we define $X^+ := X \cap {]0,\infty[}$ and $k\cdot X+h:=\{kx+h: x \in X\}$. 

We let $\HHb$ be either the integers or the non-negative integers, and we use $\mathcal P(\mathbf H)$ for the power set of $\mathbf H$. We regard $\HHb$ as a ``parameter''; though it makes
no big difference to stick to the assumption that $\mathbf H = \mathbf N$ most of the time, some
statements (e.g., Theorems \ref{thm:2.4} and \ref{thm:4.2}) will be sensitive to the actual choice of $\mathbf H$. Accordingly, we take an \emph{arithmetic progression} of $\mathbf H$ to be a set of the form $k \cdot \HHb + h$ with $k \in \NNb^+$ and $h \in \NNb$; and we denote by $\AAc_\mathbf{H}$ the family of all finite unions of arithmetic progressions of $\HHb$.

For each $k \in \mathbf N^+$ and $X \subseteq \bf Z$, we write $r_k(X)$ for the number of residues $h \in \llb 0, k-1 \rrb$ such that $X \cap (k \cdot \mathbf Z + h) \ne \emptyset$: Note that, if $X \subseteq \mathbf N$ then $r_k(X)$ is also equal to the number of residues $h \in \llb 0, k-1 \rrb$ such that $X \cap (k \cdot \mathbf N + h) \ne \emptyset$.

\section{Upper quasi-densities and small sets}
\label{sec:upper_densities}

We start with recalling the notion of ``density'' that will be used through this work.

\begin{definition}\label{def:upperdensity}
A function $\mu^\ast: \mathcal{P}(\HHb) \to \RRb$ is an \emph{upper density} \textup{(}on $\HHb$\textup{)} if it is
\begin{enumerate}[label={\rm (\textsc{f}\arabic{*})}]
\item\label{item:F1} normalized, i.e., $\mu^\ast(\HHb) = 1$;
\item\label{item:F2} monotone, i.e., $\mu^\star(X) \le \mu^\star(Y)$ for all $X,Y \subseteq \HHb$ with $X\subseteq Y$;
\item\label{item:F3} subadditive, i.e., $\mu^\ast(X \cup Y) \le \mu^\ast(X) + \mu^\ast(Y)$ for every $X,Y\subseteq \HHb$;
\item\label{item:F4} $(-1)$-homogeneous, i.e., $\mu^\ast(k\cdot X) = \frac{1}{k}\mu^\ast(X)$ for all $X\subseteq \HHb$ and $k \in \NNb^+$;
\item\label{item:F5} shift-invariant, i.e., $\mu^\ast(X+h) = \mu^\ast(X)$ for every $X\subseteq \HHb$ and $h \in \NNb$.
\end{enumerate}
In addition, we call $\mu^\ast$ an \emph{upper quasi-density} \textup{(}on $\HHb$\textup{)} if $\mu^\ast(X)\le 1$ for all $X\subseteq \HHb$ and $\mu^\ast$ satisfies \ref{item:F1} and \ref{item:F3}-\ref{item:F5}.
\end{definition}

Every upper density is obviously an upper quasi-density, and the existence of non-monotone upper quasi-densities is guaranteed by \cite[Theorem 1]{LT}. While it is arguable that non-monotone ``densities'' are not very interesting from the point of view of applications, it seems meaningful
to establish if certain properties of a specific class of objects depend or not on a particular
assumption (in the present case of interest, the axiom of monotonicity): This usually contributes to a better understanding of the objects under consideration and is basically our motivation for dealing with
upper quasi-densities instead of restricting attention to upper densities.

\begin{remark}\label{rem:examples}
	It turns out that each of the following set functions is an upper density in the sense of Definition \ref{def:upperdensity}:
\begin{enumerate}[label=$\bullet$]

\item the \textit{upper $\alpha$-density} (on $\mathbf H$), that is, the function
$$
\PPc(\HHb) \to \RRb: X \mapsto \limsup_{n \to \infty} \frac{\sum_{i \in X \cap \llb 1, n \rrb} i^{\alpha}}{\sum_{i \in \llb 1, n \rrb} i^{\alpha}},
$$
where $\alpha$ is a real parameter $\ge -1$ (most notably, $\alpha = -1$ corresponds to the upper logarithmic density, and $\alpha = 0$ to the upper asymptotic density);

\item the \textit{upper Banach} (or \textit{upper uniform}) \textit{density}, that is, the function
\[
\PPc(\HHb) \to \RRb: X \mapsto \lim_{n \to \infty} \max_{k \ge 0} \frac{|X \cap \llb k+1,k+n \rrb|}{n};
\]
\item the \textit{upper analytic density}, that is, the function
$$
\PPc(\HHb) \to \RRb: X \mapsto \limsup_{s \to 1^+} \frac{1}{ \zeta(s)} \sum_{i \, \in X^+} \frac{1}{i^s},
$$
where $\zeta$ is the restriction to the interval $]1,\infty[$ of the Riemann zeta function;

\item the \textit{upper P\'olya density}, that is, the function
$$
\PPc(\HHb) \to \RRb: X \mapsto \lim_{s \to 1^-} \limsup_{n \to \infty} \frac{|X \cap \llb 1,n \rrb| - |X \cap \llb 1, ns \rrb|}{(1-s)n};
$$

\item the \textit{upper Buck density}, that is, the function
$$
\mathfrak b_\mathbf{H}^\ast: \PPc(\HHb) \to \mathbf{R}: X \mapsto \inf_{A \in \AAc_\mathbf{H}\,:\, X \subseteq A} \dd_\mathbf{H}^\ast(A),
$$
where $\mathsf d_\mathbf{H}^\ast$ is the upper asymptotic density on $\bf H$ (see the first bullet on this list).
\end{enumerate}
We refer the reader to \cite{GrToTo10} for the existence of the limit in the above definition of the upper Banach density, and to \cite[Satz III, p.~559]{Po} for the existence of the limit in the above definition of the upper P\'olya density. For further details, see \cite[\S~2 and 
Examples 4--6 and 8]{LT}.
\end{remark}

As mentioned in the introduction, ``densities'' are mainly a technical device to formalize the idea that a set is, or is not, ``small''. This leads straight to the next definition.

\begin{definition}\label{def:Dsmall}
We say that a subset $X$ of $\HHb$ is \emph{small} if $\mu^\ast(X)=0$ for every upper quasi-density $\mu^\star$ on $\HHb$.
\end{definition}

Throughout, we will often use
the following characterization of small sets, which is basically a corollary of some of the main results of \cite[\S\S~4 and 6]{LT}.

\begin{theorem}\label{thm:2.4}
Let $X$ be a subset of $\bf H$. Then $X$ is small \textup{(}as per Definition \textup{\ref{def:Dsmall}}\textup{)} if and only if $\mathfrak b_{\bf H}^\ast(X) = 0$, if and only if $\mathfrak b_{\bf Z}^\ast(X) = 0$ \textup{(}see Remark \textup{\ref{rem:examples}} for the notation\textup{)}.
\end{theorem}
\begin{proof}
We know from \cite[Proposition 2(vi) and Theorem 3]{LT} that $0 \le \mu^\star(X) \le \mathfrak{b}^\star_{\HHb}(X)$ for every upper quasi-density $\mu^\star$ on $\HHb$. Hence, $X$ is small if and only if $\mathfrak{b}_{\bf H}^\ast(X)=0$. To complete the proof, it is therefore enough to show that $\mathfrak b^\ast_\mathbf{H}(X) = \mathfrak b^\ast_\mathbf{Z}(X)$.

To this end, note that $\mathsf d_{\bf Z}^\ast(S) = \mathsf d_{\bf Z}^\ast(S \cap \mathbf N) = \mathsf d_{\bf N}^\ast(S \cap \bf N)$ for every $S \subseteq \mathbf Z$.
On the other hand, a set $A \in \AAc_\mathbf{N}$ containing a subset $Y$ of $\bf N$ extends, in an obvious way, to a set $A^\prime \in \AAc_\mathbf{Z}$ with $Y \subseteq A^\prime$ and $\mathsf d_{\bf N}^\ast(A) = \mathsf d_{\bf Z}^\ast(A^\prime)$; and conversely, if a set $B \in \AAc_\mathbf{Z}$ contains a subset $Y$ of $\mathbf N$, then it is clear that 
\[
Y \subseteq B \cap \mathbf N,
\quad
B \cap \mathbf N \in \AAc_\mathbf{N},
\quad\text{and}\quad
\mathsf d_{\bf Z}^\ast(B) = \mathsf d_{\bf N}^\ast(B \cap \bf N). 
\]
Stitching all the pieces together, we thus conclude that
\[
\mathfrak b^\ast_\mathbf{H}(X) = \inf_{A \in \AAc_\mathbf{H}\,:\,  X \subseteq A} \mathsf d^\ast_\mathbf{H}(A) =  \inf_{B \in \AAc_\mathbf{Z}: X \subseteq B} \mathsf d^\ast_\mathbf{Z}(B) = \mathfrak b^\ast_\mathbf{Z}(X).
\] 
And this finishes the proof.
\end{proof}

As a consequence of Theorem \ref{thm:2.4}, the property of being small is independent of the choice of $\HHb$. 
In addition, since the upper Buck density on $\bf Z$ is monotone and subadditive (as is true for any upper density), we obtain:

\begin{corollary}\label{cor:idealDsmall}
	The family of small subsets of $\HHb$ is an ideal on $\HHb$. In particular, every subset of a small set is small.
\end{corollary}

Notice that Corollary \ref{cor:idealDsmall} is not obvious a priori, since it is unknown whether the zero set of an upper quasi-density on $\bf H$ is closed under taking subsets, cf.~\cite[Question 5]{LT}. 

With this said, we are going to derive an ``explicit formula'' for (a certain generalization of) the upper Buck density that 
is perhaps of independent interest in light of Theorem \ref{thm:2.4}
and the role played by the upper Buck density in the present work 
(cf.~\cite[Theorem 1]{Pas} for a weaker result along the same lines). The reader may want to review \S~\ref{sec:notation} and \cite[Example 5]{LT} before reading further.

\begin{proposition}\label{prop:2.6}
	Let $\mu^\star$ be an upper quasi-density on $\HHb$ and $\mathfrak b^\ast(\mathscr A_\mathbf{H}; \mu^\ast)$ the function
	$$
	\mathcal P(\mathbf H) \to \mathbf R: X \mapsto \inf_{A \in \mathscr A_\mathbf{H}\,:\, X \subseteq A} \mu^\ast(A).
	$$
	Moreover, let $X$ be a subset of $\HHb$ and $(k_n)_{n \ge 1}$ an increasing sequence of positive integers such that, for every $m \in \NNb^+$, $k_n$ is divisible by $m$ for all large $n \in \mathbf N^+$. Then 
	$$
	\mathfrak{b}^\ast(\mathscr A_\mathbf{H}; \mu^\ast)(X) = \mathfrak b_\mathbf{Z}^\ast(X) = \inf_{k \ge 1} \frac{r_k(X)}{k} = \lim_{n \to \infty} \frac{r_{k_n}(X)}{k_n}.
	$$
\end{proposition}

\begin{proof}
To begin, let $A \in \mathscr A_\mathbf{H}$ and suppose $X \subseteq A$. By definition, $A = \bigcup_{h \in \mathcal H} (k \cdot \mathbf H + h)$ for some $k \in \mathbf N^+$ and $\mathcal H \subseteq \llb 0, k-1 \rrb$. Accordingly, set
$$
\mathcal H^\prime := \{h \in \mathcal H: X \cap (k \cdot \mathbf H + h) \ne \emptyset\}
\quad\text{and}\quad
\textstyle A^\prime := \bigcup_{h \in \mathcal H^\prime} (k \cdot \mathbf H + h).
$$
Then $X \subseteq A^\prime \subseteq A$ and $A^\prime \in \mathscr A_\mathbf{H}$, and it is clear that $r_k(X) = |\mathcal H^\prime|$.
Hence, we obtain from \cite[Propositions 7 and 9]{LT} that
$$
\frac{r_k(X)}{k} = \frac{|\mathcal H^\prime|}{k} = \mu^\ast(A^\prime) \le \mu^\ast(A).
$$
Conversely, let $k \in \mathbf N^+$, and take $A := \bigcup_{h \in \mathcal H} (k \cdot \mathbf H + h)$, where 
$\mathcal H$ denotes the set of all residues $h \in \llb 0, k-1 \rrb$ for which $X \cap (k \cdot \mathbf H + h)$ is non-empty. 
Then $A \in \mathscr A_\mathbf{H}$ and $X \subseteq A$, and similarly as in the previous paragraph, $r_k(X) = |\mathcal H|$ and $\mu^\ast(A) = r_k(X)/k$.

So putting it all together, we can readily conclude from the above that
\begin{equation}\label{equ:alternative-for-Buck}
\mathfrak b^\ast(\mathscr A_\mathbf{H}; \mu^\ast)(X) = \inf_{k \ge 1} \frac{r_k(X)}{k}.
\end{equation}
On the other hand, it is straightforward that $r_k(X)/k \le r_h(X)/h$ for all $h, k \in \mathbf N^+$ such that $h \mid k$. Therefore, it follows from our assumptions that, for every $k \in \mathbf N^+$, the inequality 
$
r_{k_n}(X)/k_n \le r_k(X)/k
$ holds true for all but finitely many $n \in \mathbf N^+$; and this implies by  \eqref{equ:alternative-for-Buck} that $r_{k_n}(X)/k_n \to \mathfrak b^\ast(\mathscr A_\mathbf{H}; \mu^\ast)(X)$ as $n \to \infty$. 

The proof is thus complete, because the preceding conclusions are independent from the choice of the upper quasi-density $\mu^\ast$, and letting $\mu^\ast$ be the upper asymptotic density on $\mathbf H$ yields by Theorem \ref{thm:2.4} that $\mathfrak b^\ast(\mathscr A_\mathbf{H}; \mu^\ast)(X) = \mathfrak b_\mathbf{Z}^\ast(X)$.
\end{proof}

\section{Criteria for smallness and examples}\label{sec:preliminaries}

Given an upper [quasi-]density $\mu^\ast$ on $\HHb$, we follow \cite{LT} and refer to the function
$$
\mu_\ast: \PPc(\HHb) \to \RRb: X \mapsto 1 - \mu^\ast(\mathbf H \setminus X)
$$
as the \emph{lower \textup{[}quasi-\textup{]}density} on $\HHb$ conjugate to $\mu^\ast$, or simply as the \textit{conjugate} of $\mu^\ast$. 
We list below some basic properties of upper and lower [quasi-]densities.

\begin{lemma}\label{lem:basicUD1}
Let $\mu^\star$ be an upper quasi-density on $\HHb$ with conjugate $\mu_\ast$, and let $X$ be a subset of $\bf H$. The following hold:
\begin{enumerate}[label={\rm (\roman{*})}]
\item \label{basicUD1} $\mu^\star(\mathcal P(\mathbf H))=\mu_\star(\mathcal P(\mathbf H))=[0,1]$ and  $\mu_\star(X)\le \mu^\star(X)$.
\item \label{basicUD3} If $X$ is finite, then $\mu^\star(X)=0$.
\item \label{listUD4} If $k\cdot \HHb+h\subseteq X$ for some $k \in \NNb^+$ and $h \in \NNb$, then $\mu^\ast(X) \ge \frac{1}{k}$. Symmetrically, if $X\subseteq k\cdot \HHb+h$ then $\mu^\ast(X) \le \frac{1}{k}$.
\end{enumerate}
\end{lemma}
\begin{proof}
See \cite[Theorem 2, Propositions 2(vi) and 6, and Corollary 2]{LT}.
\end{proof}

The next result extends a criterion used in \cite[\S~3, p.\ 563]{Bu0} to demonstrate that the upper Buck density of the set of squares is zero; see, in particular, the corollary to \cite[Theorem 3]{Bu0}.

\begin{proposition}
\label{th:modular_criterion}
Let $\mu^\ast$ be an upper density on $\HHb$ with conjugate $\mu_\ast$, and for every $k \in \NNb^+$ and $S \subseteq \HHb$ denote by $w_k(S)$ the cardinality of the set
$$
\mathcal W_k(S) := \{h \in \llb 0, k-1 \rrb: \mu^\ast(S \cap (k \cdot \HHb + h)) \ne 0\}.
$$
Next, let $X$ be a subset of $\HHb$, and let $(k_n)_{n \ge 1}$ be a sequence of pairwise coprime positive integers. The following hold:
\begin{enumerate}[label={\rm (\roman{*})}]
\item\label{item:th:modular_criterion_2} $w_k(X) \le r_k(X)$ for all $k \in \mathbf N^+$, and for every $n \in \mathbf N^+$ we have
\begin{equation}\label{equ:equ_1}
\prod_{i=1}^n \left(1 - \frac{w_{k_i}(\mathbf H \setminus X)}{k_i}\right)\! \le \mu_\ast(X) \le \mu^\ast(X) \le \prod_{i=1}^n \frac{w_{k_i}(X)}{k_i}.
\end{equation}
\item\label{item:th:modular_criterion_3} If $\sum_{n = 1}^\infty (1 - w_{k_n}(X)/k_n) = \infty$, then $\mu^\ast(X) = 0$. In particular, if $\sum_{n = 1}^\infty k_n^{-1} = \infty$ and $w_{k_n}(X) \le k_n - 1$ for all $n \in \NNb^+$, then $\mu^\ast(X) = 0$.

\item\label{item:th:modular_criterion_4} If $\sum_{n = 1}^\infty k_n^{-1} = \infty$ and $r_{k_n}(X) \le k_n-1$ for all $n \in \mathbf N^+$, or more generally if $\sum_{n = 1}^\infty (1 - r_{k_n}(X)/k_n) = \infty$, then $X$ is small.
\end{enumerate}
\end{proposition}

\begin{proof}

\ref{item:th:modular_criterion_2} The first inequality is obvious. For the other, notice that the function $\NNb^+ \to \NNb: q \mapsto w_q(X)$ is submultiplicative, that is, $w_{mn}(X) \le w_{m}(X) w_{n}(X)$ for all $m, n \in \NNb^+$ with $\gcd(m,n) = 1$: This is so because, for all $m, n \in \mathbf N^+$, the function
$$
\mathcal{W}_{mn}(X) \to \mathcal{W}_{m}(X) \times \mathcal{W}_{n}(X): h \mapsto (h \bmod m, h \bmod n)
$$
is well defined (here is where we use that $\mu^\ast$ is monotone); and if, in addition, $m$ and $n$ are coprime, then the function is injective by the Chinese remainder theorem. 
Consequently, we get from \cite[Proposition 11]{LT} that, 
for every $n \in \NNb^+$,
$$
\mu^\ast(X) \le \frac{w_{k_1 \cdots k_n}(X)}{k_1 \cdots k_n} \le \prod_{i=1}^n \frac{w_{k_i}(X)}{k_i}
$$
and
\begin{equation*}
\begin{split}
\mu_\ast(X) & \ge 1 - \frac{w_{k_1 \cdots k_n} (\mathbf H \setminus X)}{k_1 \cdots k_n} \ge 1 - \prod_{i=1}^n \frac{w_{k_i}(\mathbf H \setminus X)}{k_i} \ge \prod_{i=1}^n \left(1 - \frac{w_{k_i}(\mathbf H \setminus X)}{k_i}\right),
\end{split}
\end{equation*}
where we have used that $1 - a_1 \cdots a_n \ge 1 - a_1 \ge (1-a_1) \cdots (1 - a_n)$ for all $a_1, \ldots, a_n \in [0,1]$. By Lemma \ref{lem:basicUD1}\ref{basicUD1}, this suffices to finish the proof.

\ref{item:th:modular_criterion_3} If $w_{k_n}(X)=0$ for some $n \in \NNb^+$, the claim follows at once from \eqref{equ:equ_1}. Otherwise, $1 \le w_{k_n}(X) \le k_n$ for all $n \in \NNb^+$. So, recalling that 
$$
\log x \le -(1-x),
\quad\text{for } x \in {]0,1]},
$$
we find that, for every $n \in \mathbf N^+$,
\begin{equation}
\label{equ:4}
\prod_{i=1}^n \frac{w_{k_i}(X)}{k_i} = \exp\left(\sum_{i=1}^n \log \frac{w_{k_i}(X)}{k_i}\right)\! \le \exp\left(-\sum_{i=1}^n \frac{1-w_{k_i}(X)}{k_i}\right).
\end{equation}
But the right-most side of \eqref{equ:4} tends to $0$ as $n \to \infty$, since we are assuming $\sum_{n = 1}^\infty (1- w_{k_n}(X)/k_n) = \infty$; and this, together with part \ref{item:th:modular_criterion_2}, implies $\mu^\ast(X) = 0$.

The rest is straightforward, because if $w_{k_n}(X) \le k_n - 1$ for every $n \in \NNb^+$, then it is clear that $\sum_{n = 1}^\infty (1 - w_{k_n}(X)/k_n) \ge \sum_{n = 1}^\infty k_n^{-1} = \infty$.

\ref{item:th:modular_criterion_4} This is immediate from parts \ref{item:th:modular_criterion_2} and \ref{item:th:modular_criterion_3} and the arbitrariness of $\mu^\ast$.
\end{proof}
We continue with a common generalization of \cite[Corollary 2]{Niv} and \cite[Lemma 2]{BF}.

\begin{proposition}
\label{prop:3.3}
Let $X \subseteq \HHb$, and let $\mu^\ast$ be an upper density on $\mathbf H$.
Moreover, assume that $(k_n)_{n \ge 1}$ is a sequence 
of pairwise coprime positive integers such that $\sum_{n = 1}^\infty k_n^{-1} = \infty$, 
and for each $n \in \NNb^+$ set $X_n := \{x \in X: k_n \mid x\}$. 
If there are only finitely many $n \in \NNb^+$ for which $\mu^\ast(X_n) > 0$, then $\mu^\ast(X) = 0$; in particular, the set $\{k_n: n \in \NNb^+\}$ is small.
\end{proposition}

\begin{proof}	
By hypothesis, the series $\sum_{n = 1}^\infty k_n^{-1}$ diverges to $\infty$ and the set 
	\[
	I := \{0\} \cup \{n \in \NNb^+: \mu^\ast(X_n) > 0\}
	\]
	is finite (and non-empty). Hence, $\sum_{n = n_0}^\infty k_n^{-1} = \infty$ and $w_{k_n}(X) \le k_n - 1$ for all $n \ge n_0$, where  $n_0 := 1 + \max I$ and $w_{k_n}(X)$ is the number of residues $h \in \llb 0, k_n - 1 \rrb$ such that $\mu^\ast(X \cap (k \cdot \mathbf H + h)) \ne 0$. By Proposition \ref{th:modular_criterion}\ref{item:th:modular_criterion_3}, it follows that $\mu^\ast(X) = 0$. 

In particular, if $X$ is the set $\{k_n: n \in \NNb^+\}$, then it is clear that $X_n = \{k_n\}$ for all $n \in \NNb^+$, and we conclude from
Lemma \ref{lem:basicUD1}\ref{basicUD3} 
and the above that $\mu^\ast(X) = 0$, which implies that $X$ is small since we may take $\mu^\ast$ to be the upper Buck density on $\bf H$.
\end{proof}

We will repeatedly resort to Propositions \ref{th:modular_criterion} and \ref{prop:3.3} to show that various sets of integers are small. We begin with a generalization of \cite[Theorems 1--3]{BF} and \cite[Theorem 1]{Bu0}; then we proceed to prove a result on the ``density'' of a factorial-like sequence (cf.~Remark \ref{rem:sparsity-does-not-imply-small}).

\begin{corollary}
\label{cor:upper_density_of_primes}
For every $k \in \mathbf N$, the 
set $X^{(k)}$ \textup{(}resp., $Y^{(k)}$\textup{)} of all integers $x \in \HHb$ that factor into a product of exactly $k$ \textup{(}resp., at most $k$\textup{)} primes \textup{(}counted with multiplicity\textup{)}, is small.
\end{corollary}

\begin{proof}
Fix $k \in \mathbf N$. By subadditivity, it suffices to show that $X^{(k)}$ is small, since we have that $Y^{(k)} = X^{(0)} \cup \cdots \cup X^{(k)}$.
We argue by induction on $k$. 

In light of Lemma \ref{lem:basicUD1}\ref{basicUD3}, the claim is trivial if $k = 0$, because $X^{(0)} = \mathbf{H} \cap \{\pm 1\}$. 
Accordingly, assume that $k$ is a positive integer and $X^{(k - 1)}$ is small, and let $p \in \mathbf P^+$. Then it is clear that
$$
X_p^{(k)} := \{x \in X^{(k)}: p \mid x\} = p \cdot X^{(k-1)},
$$
and this implies by \ref{item:F4} and the inductive hypothesis that $\mu^\ast(X_p^{(k)}) = 0$ for every upper quasi-density $\mu^\ast$ on $\HHb$. Therefore, we can conclude from Proposition \ref{prop:3.3}, applied with $k_n$ equal to the $n$-th prime of $\mathbf N^+$, that also $X^{(k)}$ is small, since it is well known (see, e.g., \cite[Theorem 1.13]{Apo}) that $\sum_{p \, \in \, \PPb^+} \frac{1}{p} = \infty$.
\end{proof}

\begin{corollary}
\label{prop:density_of_factorial_sequences}
Let $(x_n)_{n \ge 1}$ be a sequence in $\HHb$ with the property that $x_n$ divides $x_{n+1}$ for each $n$. Then the set $X := \{x_n: n \in \NNb^+\}$ is small.
\end{corollary}

\begin{proof}
By Lemma \ref{lem:basicUD1}\ref{basicUD3}, we can assume without loss of generality that the sequence $(x_n)_{n\ge 1}$ consists of pairwise distinct elements and $|x_n| \le |x_{n+1}|$ for all $n \in \mathbf N^+$. In particular, this ensures that $x_1 \neq 0$ and $|x_{2n}|<|x_{2n+2}|$ for every $n$.
  
Then $r_{|x_{2n}|}(X) \le 2n$ for every $n$, since $x_h \mid x_k$ for all $h, k \in \mathbf N^+$ such that $h \mid k$. On the other hand, it is easy to verify (by induction) that $|x_{2n}| \ge 2^{n-1}$ for all $n$. So, we obtain from Proposition \ref{th:modular_criterion}\ref{item:th:modular_criterion_2} that
\begin{equation*}
\mu^\ast(X) \le \inf_{n \ge 1} \frac{r_{|x_{2n}|}(X)}{|x_{2n}|} \le \liminf_{n \to \infty} \frac{2n}{2^{n-1}} = 0,
\end{equation*}
for every upper quasi-density $\mu^\ast$ on $\bf H$. In other terms, $X$ is small.
\end{proof}

In particular, it follows from Corollaries \ref{cor:idealDsmall} and \ref{prop:density_of_factorial_sequences} that a set $X \subseteq \HHb$, whose elements are factorials, primorials, or numbers of the form $a^k$ for some fixed base $a \in \HHb$, is small. This is further strengthened by the next result, which is also a generalization of the unnumbered corollary after Theorem 3 in \cite[p. 565]{Bu0}.

\begin{corollary}
\label{cor:upper_density_of_perfect_powers}
The set $X := \bigcup_{n \ge 2} \{a^n: a \in \HHb\}$ is small.
\end{corollary}

\begin{proof}
Let $p \in \mathbf P^+$ and pick an element $x \in X$. It is clear that $p \mid x$ if and only if $p^2 \mid x$.
Hence, $r_{p^2}(X) \le p^2 - p + 1 \le p^2 - \frac{1}{2}p$. It follows (cf.~Corollary \ref{cor:upper_density_of_primes}) that
$$\sum_{p \,\in\, \PPb^+} \left(1 -\frac{r_{p^2}(X)}{p^2}\right)\! \ge \frac{1}{2} \sum_{p \, \in \, \PPb^+}
\frac{1}{p} = \infty.
$$
So we can conclude from Proposition \ref{th:modular_criterion}\ref{item:th:modular_criterion_4}, applied with $k_n$ equal to the square of the $n$-th prime of $\mathbf N^+$, that $X$ is small.
\end{proof}

We conclude our series of corollaries with a result on ``digit representations'', herein regarded as words in the free monoid over $\llb 0, b-1 \rrb$ for a fixed base $b \ge 2$. 

\begin{corollary}\label{thm:digits}
	Given $b \in \mathbf N_{\ge 2}$, let $\mathfrak{s} = (s_1, \ldots, s_k)$ be a non-empty sequence of length $k \in \mathbf N^+$ with entries in $\llb 0, b-1 \rrb$. Then the set $X$ of all $x \in \HHb$ which do not have the word $s_1 \cdots s_k$ appearing in their base-$b$ representation, is small.
\end{corollary}

\begin{proof}
Let $n \in \mathbf N^+$. Obviously, $r_{b^{nk}}(X)$ is bounded above by the number of residues $h \in \llb 0, b^{nk} - 1 \rrb$ whose base-$b$ representation does not contain the word $s_1 \cdots s_k$, or equivalently by the number of sequences $(a_0, \ldots, a_{nk-1}) \in \llb 0, b-1 \rrb^{nk}$ with $(a_i, \ldots, a_{i + k - 1}) \ne \mathfrak s$ for every $i \in \llb 0, k-1 \rrb$. It follows $r_{b^{nk}} \le (b^k - 1)^n$; whence
	$r_{b^{nk}}(X)/b^{nk} \to 0$ as $n \to \infty$. So, by Proposition \ref{th:modular_criterion}\ref{item:th:modular_criterion_2}, $X$ is small.
\end{proof}

\begin{remark}\label{rem:sparsity-does-not-imply-small}
Based on the previous results, one might be drawn to think that every ``sufficiently sparse'' set of integers is small. However, we have from Proposition \ref{prop:2.6} that the upper Buck density of the set $X := \{h! + h: h \in \NNb\}$ equals $1$ (no matter whether $\mathbf H = \mathbf N$ or $\mathbf H = \mathbf Z$), as it is easily seen that $r_k(X) = k$ for every $k \in \mathbf N^+$.
\end{remark}
The next theorems are about integral polynomials in one variable; it could be interesting to extend them to more general classes of integer-valued functions (see \S~\ref{sec:homogeneous} for a first step in this direction).

\begin{theorem}\label{thm:polynomialandDsmallsets}
Let $F: \mathbf Z \to \mathbf Z$ be a polynomial function with coefficients in $\HHb$. Then $F(\HHb)$ is small if and only if $\deg F \ne 1$.
\end{theorem}

\begin{proof}
If $F$ is constant, then its image is small, by Lemma \ref{lem:basicUD1}\ref{basicUD3}. If, on the other hand, $F$ is of degree $1$, then there exist $a,b \in \HHb$ with $a\neq 0$ such that $F(x) = ax+b$ for all $x \in \HHb$; and this in turn implies that, for every upper quasi-density $\mu^\ast$ on $\HHb$,
$$
\mu^\star(F(\HHb))=\mu^\star(a\cdot \HHb+b)=\frac{1}{|a|}>0.
$$
Accordingly, assume hereafter that $\deg F \ge 2$. Then a well-known theorem of Frobenius (see, e.g., \cite[p. 32]{MR1395088}) ensures that the set $P_F$ of primes $p \in \PPb^+$ for which $F$ has at least two roots modulo $p$, has non-zero \textit{Dirichlet density}, meaning that the limit
$$
\lim_{s \to 1^+} \frac{\sum_{p \,\in\, P_F} 1/p^s}{\sum_{p \,\in\, \mathbf P^+} 1/p^s}
$$
exists and is positive. It follows (by the monotone convergence theorem for series) that $\sum_{p \in P_F} 1/p = \infty$, because $\sum_{p \in \mathbf P^+} 1/p = \infty$ (cf.~Corollary \ref{cor:upper_density_of_primes}); and since $r_p(F(\ZZb)) \le p-1$ for every $p \in P_F$, we conclude from Proposition \ref{th:modular_criterion}\ref{item:th:modular_criterion_4} that $F(\mathbf H)$ is small.
\end{proof}

\begin{theorem}\label{thm:polynomialprimeandDsmall}
Let $F: \mathbf Z \to \mathbf Z$ be a non-constant polynomial function with integer coefficients. Then the set $X := \{k \in \HHb: F(k) \in \PPb\}$ is small.
\end{theorem}
\begin{proof}
Let $\mu^\ast$ be an upper quasi-density on $\bf H$, and for each $n \in \mathbf N^+$ denote by $w_n(X)$ the number of residues $h \in \llb 0, k-1 \rrb$ such that $\mu^\ast(X \cap (k \cdot \mathbf H + h)) \ne 0$.

Similarly as in the proof of Theorem \ref{thm:polynomialandDsmallsets}, 
there is a set $P_F \subseteq \mathbf P^+$ such that $\sum_{p\, \in\, P_F} 1/p = \infty$ and $F$ has at least one zero modulo $p$ for every $p \in P_F$, that is, $p \mid F(h_p)$ for some $h_p \in \llb 0, p-1 \rrb$. 
In particular, it follows that, for each $p \in P_F$, the set $X \cap (p \cdot \mathbf H + h_p)$ is finite (otherwise, $|F(pk + h_p)| = p$ for infinitely many $k \in \mathbf H$, in contradiction to the fact that $F$ is non-constant), and hence, by Lemma \ref{lem:basicUD1}\ref{basicUD3}, $w_p(X) \le p-1$.

So, putting it all together, we get from Proposition \ref{th:modular_criterion}\ref{item:th:modular_criterion_2} that $\mu^\ast(X) = 0$, and this is enough to show that $X$ is small (since $\mu^\ast$ was arbitrary).
\end{proof}

\section{Binary quadratic forms}
\label{sec:homogeneous}
It is folklore that the asymptotic density of the set of integers that can be written as a sum of two squares is zero. In the present section, we generalize this to binary quadratic forms, while replacing the asymptotic density with an arbitrary upper quasi-density.

\begin{lemma}\label{lem:jacobi}
	Let $d$ be an integer, but not a perfect square. Then there exist $m \in \NNb^+$ and $r \in \NNb$ with $\gcd(m,r)=1$ such that, for every prime $p \in \mathbf P^+$ with $p \equiv r \bmod m$, $d$ is not a quadratic residue modulo $p$.
\end{lemma}
\begin{proof}
	Write $d = 2^k t^2 u\, \varepsilon $, where $t$ and $u$ are odd positive integers, $u$ is squarefree, $k$ is a non-negative integer, and $\varepsilon$ is the sign of $d$ (i.e., $\varepsilon = 1$ if $d \ge 1$, and $\varepsilon = -1$ otherwise). 
	
	If $u = 1$, then it is sufficient to consider that, for every odd prime $p \in \mathbf P^+$,
	we have from \cite[Theorems 9.9(a), 9.5, and 9.10]{Apo} that
	$$
	\left(\frac{d}{p}\right) = \left(\frac{2^k}{p}\right)\! \left(\frac{\varepsilon}{p}\right) = (-1)^{\frac{1}{8}k(p^2-1)} \varepsilon^{\frac{1}{2}(p-1)},
	$$
	where $\left(\frac{\cdot}{\cdot}\right)$ is a Jacobi symbol, see \cite[\S~9.7]{Apo}; whence we can take $p \equiv 5 \bmod 8$ if $k$ is odd, and $p \equiv 3 \bmod 4$ if $k$ is even (note that, in the latter case,  $\varepsilon$ must be equal to $-1$, or else $d$ would be a perfect square).
	
	Thus, assume from now on that $u \ge 3$. Then $u = q_1 \cdots q_n$, where $q_1, \ldots, q_n \in \mathbf P^+$ are pairwise distinct odd prime numbers; and it follows by \cite[Theorems 9.1]{Apo} that there exist $r_1, \ldots, r_n \in \mathbf N$ with 
	\[
	\left(\frac{r_1}{q_1}\right) = -1 
	\quad\text{and}\quad
	\left(\frac{r_i}{q_i}\right) = 1 
	\text{ for each }i \in \llb 2, n \rrb.
	\]
	As a result, we conclude from the Chinese remainder theorem and \cite[Theorem 9.9(c)]{Apo} 
	that there is $r \in \NNb$ with the property that
	$$
	\left(\frac{r}{q_1}\right) = -1
	\quad\text{and}\quad
	\left(\frac{r}{q_i}\right) = 1\text{ for each }i \in \llb 2, n \rrb.
	$$
	So, letting $s \in \NNb$ be such that $8s + 1 \equiv r \bmod u$ (this is possible because $u$ is odd), we have from \cite[Theorem 9.9, parts (c) and (b)]{Apo} that
	\begin{equation}
	\label{equ:jacobis}
	\left(\frac{8s + 1}{u}\right) = \bigg(\frac{r}{u}\bigg) = \prod_{i=1}^n \left(\frac{r}{q_i}\right) = -1,
	\end{equation}
	and this implies in particular that $8s+1$ and $u$ are coprime.
	
	Consequently, if $p \equiv 8s + 1 \bmod 8u$, then $\gcd(2u,p) = 1$ and $p \equiv 1 \bmod 8$;
	and we get from \cite[Theorems 9.5, 9.9(a), 9.9(d), 9.10 and 9.11]{Apo} that
	$$
	\left(\frac{d}{p}\right) = \left(\frac{2^k}{p}\right)\! \left(\frac{\varepsilon}{p}\right)\! \left(\frac{u}{p}\right) = (-1)^{\frac{1}{8}k(p^2-1)+\frac{1}{4}(p-1)(u-1)}\varepsilon^{\frac{1}{2}(p-1)} \bigg(\frac{p}{u}\bigg) =
	\left(\frac{8s + 1}{u}\right),
	$$
	which, together with \eqref{equ:jacobis}, yields $\left(\frac{d}{p}\right) = -1$ and completes the proof.
\end{proof}

\begin{theorem}
\label{thm:4.2}
Let $\mu^\ast$ be an upper quasi-density on $\HHb$, and set $X := \{ax^2 + bxy + cy^2: x, y \in \HHb\}$ and $D := b^2 - 4ac$, where $a,b,c \in \HHb$ are fixed. The following hold:
\begin{enumerate}[label=\textup{(\roman{*})}]
\item If $D$ is not a perfect square or $D = 0$, then $X$ is small.
\item If $D$ is a non-zero perfect square and either $ac = 0$ or $\mathbf H = \bf Z$, then $\mu^\ast(X) > 0$.
\end{enumerate}
\end{theorem}

\begin{proof}
Let $w_k(\cdot)$ be defined as in Proposition \ref{th:modular_criterion}. 
We have several cases.

\vskip 0.2cm
\textsc{Case 1}. $D$ is not a perfect square (and hence $ac \ne 0$). In light of Lemma \ref{lem:jacobi}, there are $m \in \NNb^+$ and $r \in \NNb$ such that, for every $p \in \mathbf P^+$ with $p \equiv r \bmod m$, $D$ is not a quadratic residue modulo $p$. Accordingly, let 
$P$ be the set of all primes $p \equiv r \bmod m$ such that $p \ge 2 + \max(|a|, |b|, |c|) \ge 3$.

    If $p \in P$, $z \in \mathbf Z$, and $p \nmid z$, then $ax^2 + bxy + cy^2 \ne pz$ for all $x, y \in \bf Z$: Otherwise, $(2ax + by)^2 \equiv Dy^2 \bmod p$, which is only possible if $p \mid y$ (since $D$ is not a quadratic residue modulo $p$); consequently, we find that $p \mid 2ax$, and hence $p \mid x$ (because $p \nmid 2a$); this, however, means that $p \mid ax^2 + bxy + cy^2$, contradicting that $p \nmid z$.

    Thus, $w_{p^2}(X) \le r_{p^2}(X) \le p^2 - p + 1 \le p^2 - \frac{1}{2}p$ for every $p \in P$, which implies,
    by Theorem \ref{th:modular_criterion}\ref{item:th:modular_criterion_3}, that $\mu^\ast(X) = 0$, since we have from Dirichlet's theorem on primes in arithmetic progressions (see, e.g., \cite[Corollary 4.12(c)]{MoVa06}) that
    $$
    \sum_{p \in P} \left(1 - \frac{w_{p^2}(X)}{p^2}\right)\! \ge \frac{1}{2}\sum_{p \in P} \frac{1}{p} = \infty.
    $$

\textsc{Case 2}: $D = ac = 0$. Note that $b = 0$, and assume by symmetry that $c = 0$. It follows that $X =  \{a x^2: x \in \HHb\}$, and Corollaries \ref{cor:idealDsmall} and \ref{cor:upper_density_of_perfect_powers} yield $\mu^\ast(X) = 0$.

\vskip 0.2cm
\textsc{Case 3}: $D$ is a non-zero perfect square and $ac = 0$. We have $|b| = \sqrt{D} > 0$, and it is immediate that $X \supseteq |b| \cdot \HHb + a+c$. Hence $\mu^\ast(X) \ge |b|^{-1} > 0$, by Lemma \ref{lem:basicUD1}\ref{listUD4}.

\vskip 0.2cm
\textsc{Case 4}: $D = q^2$ for some $q \in \NNb$ and $ac \ne 0$. Let $\varepsilon$ be the sign of $a$, and observe that $b - q \ne 0$
    and, by axiom \ref{item:F4}, $\mu^\ast(X) = 4|a|\, \mu^\ast(4|a| \cdot X)$. Next, notice that
    \begin{equation}
    \label{equ:alternative_representation_of_4aX}
    \begin{split}
    4|a| \cdot X 
    	& = \{(2ax + (b-q)y)(2ax + (b+q)y)\varepsilon: x,y \in \HHb\}.
    \end{split}
    \end{equation}
    Building on these premises, we distinguish two subcases.
 \vskip 0.2cm
 \emph{Case 4.1:} $q = 0$. 
    It is clear from \eqref{equ:alternative_representation_of_4aX} that $4|a| \cdot X \subseteq \{x^2\varepsilon: x \in \HHb\}$, and we derive from Corollaries \ref{cor:idealDsmall} and \ref{cor:upper_density_of_perfect_powers} that $\mu^\ast(X) =  4|a| \, \mu^\ast(4|a| \cdot X) = 0$.

\vskip 0.2cm
\emph{Case 4.2:}
    $q \ne 0$ and $\HHb = \ZZb$. Denote by $Y$ the set
	$$
	\{(-2a(b-q)z + 2a(b-q)(z+1))(-2a(b-q)z + 2a(b+q)(z+1))\varepsilon: z \in \ZZb\}.
	$$
    By \eqref{equ:alternative_representation_of_4aX}, $Y$ is a subset of $4|a| \cdot X$, and we see that
    \begin{equation*}
    \label{equ:the_special_case_when_H=Z}
    \begin{split}
    Y & = 4a^2 \cdot \{2q (b-q) \varepsilon z + (b^2-q^2)\varepsilon: z \in \ZZb\} \\
    & = 8a^2q\, |b-q| \cdot \ZZb + 4a^2(b^2-q^2) \varepsilon.
    \end{split}
    \end{equation*}
    In other words, $4|a| \cdot X$ contains an arithmetic progression of $\bf Z$, which, along with 
    Lemma \ref{lem:basicUD1}\ref{listUD4} 
    and the above, implies $\mu^\ast(X) = 4|a|\,\mu^\ast(4|a| \cdot X) > 0$.\qedhere
\end{proof}

\begin{remark}
	Set $X := \{ax^2 + bxy + cy^2: x, y \in \mathbf N\}$, where $a, b, c \in \mathbf N$, $ac \ne 0$, and $b^2 - 4ac = q^2$ for some $q \in \mathbf N^+$. This case is not covered by Theorem \ref{thm:4.2}, and it turns out that $\mu^\ast(X)$ is zero for some choices of the upper density $\mu^\ast$ and positive for others.
	
	In fact, we will prove that $\mathsf{d}_\NNb^\ast(X) = 0 \ne \mathfrak b_\NNb^\ast(X)$. To begin, it is easy to check that
	$$
	4a \cdot X = \{(2ax + (b-q)y)^2 + 2 q y: x,y \in \NNb\}.
	$$
	Consequently, we have that $2a(b-q) \cdot A \subseteq 4a \cdot X \subseteq B$,
	where
	$$
	A := \{2a(b-q)(u+v)^2 + 4aq v: u, v \in \mathbf N\} \subseteq \mathbf N
	$$
	and
	$$
	B := \{xy: x, y \in \mathbf N \text{ and } x \le y \le (2q+1) x\} \subseteq \mathbf N.
	$$
	Since $\mathsf d_\NNb^\ast$ and $\mathfrak b_\NNb^\ast$ are upper densities (and hence satisfy \ref{item:F2} and \ref{item:F4}), it follows that
	$$
	\mathfrak b_\NNb^\ast(X) = 4a\, \mathfrak b_\NNb^\ast(4a \cdot X) \ge 
	2(b-q)^{-1} \mathfrak b_\NNb^\ast(A) 
	\quad\text{and}\quad
	\mathsf d_\NNb^\ast(X) \le \mathsf d_\NNb^\ast(B).
	$$
	Thus, it suffices to show that $\mathfrak b_\NNb^\ast(A) \ne 0 = \mathsf d_\NNb^\ast(B)$. To this end, fix $k \in \mathbf N^+$. We have
	$$
	\{2a(b-q)(ku  + 1)^2 + 4aq ((k-1)u+1): u \in \mathbf N\} \subseteq A,
	$$
	which means that $r_k(A)$ is larger than or equal to the number of residues $h \in \llb 0, k - 1 \rrb$ such that $2a(b-q) + 4aq - 4aq u \equiv h \bmod k$ for some $u \in \mathbf N$. Then $r_k(A) \ge \lfloor (4aq)^{-1} k \rfloor$, and we conclude from Proposition \ref{prop:2.6} that $\mathfrak b_\NNb^\ast(A) \ge (4aq)^{-1} > 0$.
	
	As for the rest, let $n \in \mathbf N^+$ and pick $z \in B \cap \llb 1, n \rrb$. By construction, $z = xy$ for some $x, y \in \mathbf N^+$ with $x \le y \le (2q+1) x$. Therefore, we have $
	y^2 \le (2q+1) xy \le (2q+1)n$,
	and hence 
	$
	x \le y \le \sqrt{(2q+1)n}$. In other words, $B \cap \llb 1, n \rrb$ is a subset of the mul\-ti\-pli\-ca\-tion table for positive integers $\le \sqrt{(2q+1)n}$. So, we get from a classic result of Erd\H{o}s \cite[Part 3]{Erd1955} that $\bigl|B \cap \llb 1, n \rrb\bigr| = o(n)$ as $n \to \infty$, which yields $\mathsf d_\NNb^\ast(B) = 0$.
\end{remark}

\section{Closing remarks and open questions}\label{sec:closings}
We conclude the paper with a couple of open questions naturally stemming from the results of \S\S~\ref{sec:upper_densities} and \ref{sec:homogeneous} that we have not been able to settle.

\begin{question}\label{q:5.1}
	By Corollary \ref{cor:idealDsmall} and Lemma \ref{lem:basicUD1}\ref{basicUD3}, the family $\mathcal{S}$ of all small subsets of $\bf H$ is an ideal containing all finite subsets of $\mathbf H$; and by Theorem \ref{thm:2.4}, it coincides with the intersection of $\bf H$ and the zero set of the upper Buck density on $\mathbf Z$. One may wonder if $\mathcal{S}$ is also ``closed under products'', meaning that 
	$$
	XY := \{xy: (x,y) \in X \times Y\} \in \mathcal{S},
	\quad \text{for all }X, Y \in \mathcal{S}.
	$$
	The answer is negative. 
	Indeed, let $\mu^\ast$ be an upper quasi-density on $\HHb$, and let $A$ (resp., $B$) be the set of all $x \in \HHb$ whose positive prime divisors are all equal to $1$ (resp., $3$) modulo $4$.
	Then 
	$\{x \in A: p \mid x\}$ is empty for every $p \in \mathbf P^+$ with $p \equiv 3 \bmod 4$; and in a similar way, $\{y \in B: q \mid y\} $ is empty for every $q \in \mathbf P^+$ with $q \equiv 1 \bmod 4$ (observe that $0 \notin A \cup B$, because $k \mid 0$ for all $k \in \mathbf N$). Therefore, we get from
	Dirichlet's theorem on primes in arithmetic progressions, Lemma \ref{lem:basicUD1}\ref{basicUD3}, and Prop\-o\-si\-tion \ref{prop:3.3} that both $A$ and $B$ are small. But $AB = 2 \cdot \mathbf H + 1$ (note that $\mathbf H \cap \{\pm 1\} \subseteq A \cap B$), and we have by \ref{item:F4} that $\mu^\ast(2 \cdot \mathbf H + 1) = \frac{1}{2}$.

	Likewise, $\mathcal{S}$ is not ``closed under sums'' either, in the sense that there is $X \in \mathcal{S}$ such that $X+X := \{x+y: x, y \in X\} \notin \mathcal{S}$. In fact, the set $\mathcal Q :=\{x^2+y^2: x,y \in \HHb\} \subseteq \bf H$ is small by Theorem \ref{thm:4.2}; but $\mathcal Q + \mathcal Q = \mathbf N$ (by Lagrange's four square theorem), and hence the upper Buck density of $\mathcal Q + \mathcal Q$ is $1$ (regardless of whether $\mathbf H = \mathbf N$ or $\mathbf H = \mathbf Z$). 
	
	So, we may ask if, for a given $\alpha \in [0,1]$, there exists a set $X \in \mathcal{S}$ such that $\mu_\ast(XX) = \mu^\ast(XX) = \alpha$ (resp., $\mu_\ast(X+X) = \mu^\ast(X+X) = \alpha$) for every upper quasi-density $\mu^\ast$ on $\HHb$, where $\mu_\ast$ is the conjugate of $\mu^\ast$ (as defined in the first lines of \S~\ref{sec:preliminaries}). Questions in a similar vein have been recently answered in \cite{LT20}. 

	Questions in the same vein have been recently addressed by various authors in the special case of the upper asymptotic density, see \cite[Theorem 1.3 and Question Q4]{Fa-Gr-Pa-So-2018}, \cite[\S\S~2 and 4]{Heg-Hen-Pa-2019}, and \cite[Theorems 1.1.a), 1.5, and 1.8]{Bi-Hen2019}.
\end{question}

\begin{question}
	Let $k \in \mathbf N$. By Corollary \ref{cor:upper_density_of_primes}, the set $Y^{(k)}$ of all $x \in \HHb$ that factor into a product of at most $k$ primes (counted with multiplicity), is small. Does the same hold true for the inverse image of $\llb 0, k \rrb$ under the function $\omega: \mathbf H \setminus \{0\} \to \mathbf N$ that maps a non-zero integer $x \in \HHb$ to the number of primes $p \in \mathbf P^+$ such that $p \mid x$? If yes, this would be a stronger result than Corollary \ref{cor:upper_density_of_primes}, because every subset of a small set is small (Corollary \ref{cor:idealDsmall}), and it is clear that $Y^{(k)} \subseteq \omega^{-1}(\llb 0, k \rrb)$.
\end{question}

\section*{Acknowledgments}

We thank the anonymous referees whose comments helped to improve the paper.

\end{document}